\documentclass[11pt]{article}
\usepackage{amsthm, amsfonts,epsf,amsmath,amssymb,tikz,url,listings}
\usepackage{multirow}

\begin{document}
\sloppy
\title{There is no $(95,40,12,20)$ strongly regular graph}

\author{
Jernej Azarija $^{a}$
\and
Tilen Marc $^{a}$
}

\date{\today}

\maketitle
\begin{center}
$^a$ Institute of Mathematics, Physics and Mechanics, Ljubljana, Slovenia \\
{\tt jernej.azarija@gmail.com}
\\
{\tt tilen.marc@imfm.si}
\medskip
\end{center}

\begin{abstract}
We show that there is no $(95,40,12,20)$ strongly regular graph and, consequently, there is no $(96,45,24,18)$ strongly regular graph, no two-graph on $96$ vertices, and no partial geometry $\rm{pg}(5,9,3)$. The main idea of the result is based on the star complement technique and requires a small amount of computation.
\end{abstract}

\newtheorem{theorem}{Theorem}
\newtheorem{prop}{Proposition}
\newtheorem{lemma}{Lemma}
\newcommand{\comp}{\rm{Comp}(H,\lambda)}

\section{Introduction}

A $k$-regular graph of order $v$ is said to be $(v,k,\lambda, \mu)$ {\em strongly regular} (SRG in short) if for any two distinct vertices $x,y \in V(G)$ the intersection $N(x) \cap N(y)$ has cardinality $\lambda$ if $x$ and $y$ are adjacent and $\mu$ otherwise. While there are certain necessary conditions on $v,k,\lambda,\mu$ there is no general way to determine whether a $(v,k,\lambda,\mu)$ strongly regular graph exist for a given choice of $v,k,\lambda$ and $\mu$. For example, at the time of writing this article, the existence of a $(65,32,15,16)$ SRG is not settled and is in fact the smallest case for which the existence of a SRG is not known. See Brouwer's web-page \cite{Brower} for the current state of affairs on the classification of SRG's as well as \cite{cohen-2016} for the current initiative of bringing Brouwer's classification into Sage. 

Given the difficulty of providing a general answer to the existence of strongly-regular graphs, many results targeting specific parameters accumulated over the years. For example, Haemers showed~\cite{haemers} that there is no $(76,21,2,7)$ SRG while Degraer ruled out~\cite{degr} the existence of a  $(96,38,10,18)$ SRG. The problem of determining the existence of $(75,32,10,16)$ and $(95,40,12,20)$ strongly-regular graphs was given a particular significance in \cite{godsil2001algebraic} for two reasons. On one hand, they fall into the small class of unresolved cases having less than $100$ vertices. Second, the (non-)existence of $(95,40,12,20)$ strongly-regular graphs determines the (non-)existence of $(96,45,24,18)$ strongly regular graphs as well as of the (non-)existence of so called two-graphs on 96 vertices  \cite[Chapter 11]{godsil2001algebraic}. Additional significance is given by the property that the non-existence of this graphs implies non-existence of so called partial geometry $\rm{pg}(5,9,3)$ (for the details, see \cite{Brower2}).  Certain structural results for this graphs were obtained through the years. In particular, for a $(95,40,12,20)$ SRG $X$ it was shown by Makhnev~\cite{Makhnev} that it does not contain a $20$-regular subgraph while Behbahani and Lam~\cite{Behbahani} showed that the only prime divisors of $|{\rm Aut}(X)|$ are $2,3$ and $5$. A recent result also showed that an independent set of $X$ can only have cardinality $18$. Indeed, if $X$ had an independent set $S$ of size $19$, then the graph $X \setminus S$ would have been a $(76,30,8,14)$ SRG (see Theorem 9.4.1 in \cite{brouwer3}), but the non-existence of a $(76,30,8,14)$ SRG was recently settled by Bondarenko, Prymak and Radchenko~\cite{bono}.

In \cite{self-cite} we showed that there is no $(75,32,10,16)$ strongly-regular graph. In this paper we continue our work and use the developed approach to show the infeasibility of the parameter $(95,40,12,20)$.  The outline of the proof is surprisingly similar to what is presented in \cite{self-cite} and therefore, for the sake of brevity, we give a slightly shorter overview of the tools that we use, and refer the reader to~\cite{self-cite} for additional references.

\section{Preliminaries} 

In what follows $X$ will denote a $(95,40,12,20)$ strongly regular graph. The idea of our approach is based on three steps. First, we build a list of graphs $\mathcal{L}$ such that at least one member of $\mathcal{L}$ is an induced subgraph of $X$. In addition every graph in $\mathcal{L}$ has precisely 20 vertices and does not have $2$ as an eigenvalue - we call such graphs {\em star complements}. In the second part we compute the so called {\em comparability graphs} for graphs in $\mathcal{L}$. For our purposes a comparability graph is a graph with vertex set 
$$ V(\rm{comp}(H)) = \{  u \in \{0,1\}^{n-k} \mid \langle u, u \rangle = 2 \mbox{ and } \langle u, \overrightarrow{1} \rangle = -1\}\,, $$ 
and adjacency defined as $$u \sim v \iff \langle u, v \rangle \in \{-1,0\}\,,$$ 
where $A_H$ is the adjacency matrix of $H$. Finally, for every such comparability graph we show that its clique number is smaller than $75$. The theory of star complements then guarantees that $X$ does not exist. See~\cite{Star} for a general overview of the star complement technique and \cite{Milosevic} for an application of this approach for classifying $(57,14,1,4)$ SRG's. 

In order to be able to build a small enough list $\mathcal{L}$ we use a well-known interlacing criterion explained in~\cite{self-cite}. This is possible since eigenvalues of strongly-regular graphs are easy to compute: in particular, $X$ must have eigenvalues $40, 2,-10$ of multiplicities $1,75,19$, respectively, see \cite{Brower}.

Since the interlacing criterion is by far not sufficient, we additionally need to prove certain structural results about $X$. This is done in the next section. Some of the lemmas that follow, rely on small Sage programs that generate certain induced structure of $X$. Every statement indicating that it was obtained computationally is marked in Table \ref{table:sageprog} that also lists the name of the corresponding Sage program used in the proof. All the programs and source code used in the paper can be obtained on the author's GitHub page~\cite{GitHub}.

\section{Clique structure of $X$}

The well known Hoffman's inequality \cite[pp. 204]{godsil2001algebraic} bounds the size of a independent set in a regular graph with a given least eigenvalue. In particular for our SRG $X$, Hoffman's bound implies that $\overline{X}$ must have maximal independent set of size 5, thus the clique number of $X$ is at most $5$. In addition, the recent bound on the number of $4$-cliques of a strongly-regular graph \cite{bono} implies that the clique number of $X$ is either $4$ or $5$. In this section we will show the following:

\begin{prop} \label{prop:clnum}
If $X$ exists, its clique number is $5$. Moreover, every $4$-clique of $X$ is contained in a $5$-clique.
\end{prop}

Suppose that $X$ has a $4$-clique $K_4$ that is not contained in a $5$-clique of $X$ and let $\vec{b} = (b_0,b_1,b_2,b_3)$ be a vector where for $0 \leq i \leq 3$ we denote by $b_i$ the number of vertices of $V(X) \setminus V(K_4)$ having precisely $i$ neighbors in $K_4$. As it was explained and used in \cite{self-cite}, a formula from \cite{bono} gives possible candidates for $\vec{b}$. In particular, in graph $X$ it must hold  $\vec{b} \in \{(3, 28, 60, 0),(1,34,54,2),(2,31,57,1),(0,37,51,3) \}$. In the next four subsections we analyze each possibility for $\vec{b}$ showing that a $(95,40,12,20)$ SRG with such a configuration does not exists. 

\subsection{Case $(3, 28, 60, 0)$}

Let us denote with $X_0, X_1, X_2$ the subsets of vertices in $V(X)\setminus V(K_4) $ that have, respectively, 0, 1, 2 neighbors in $K_4$. Moreover, denote the vertices in  $X_0$ by $x_1,x_2,x_3$.

\begin{lemma} Every vertex in $X_2$ has precisely two neighbors in $X_0$.
\end{lemma}

\begin{proof}
Let $x_i\in X_0$. We use an argument that will be repeatedly used in this paper. Since $x_i$ is not adjacent to any of the vertices in $K_4$ it has to have 20 common neighbors (since $X$ is strongly regular with $\mu=20$) with each vertex of $K_4$. Thus there are $4\cdot 20$ paths of length 2 from $x_i$ to $K_4$. On the other hand, $x_i$ has 40 neighbors ($X$ is 40-regular) in $X_0 \cup X_1 \cup X_2$. All the neighbors are in fact in $X_2$, for otherwise they could not form $80$ $2$-paths to $K_4$. Since for $j \in \{1,2,3\} \setminus \{i\}$ we have $|N(x_i)\cap N(x_j)|=20$ and $|N(x_i)|=40$, it follows 
$$60\geq |N(x_1) \cup N(x_2) \cup N(x_3)| = 3\cdot 40- 3\cdot 20 + |N(x_1) \cap N(x_2) \cap N(x_3)|\,,$$

by the inclusion-exclusion principle. Thus $N(x_1) \cap N(x_2) \cap N(x_3) = \emptyset$. Therefore every vertex in $X_2$ is adjacent to precisely two vertices in $X_0$.
\end{proof}

For $1\leq i<j\leq 3$ let $X_{i,j} \subseteq X_2$ be the set $N(x_i) \cap N(x_j)$. By the previous lemma, this sets are disjoint of order 20.

\begin{lemma}
Each of the graphs $X[X_{1,2}],X[X_{1,3}]$ and $X[X_{2,3}]$ is a disjoint union of cycles. Moreover the graph $X[X_{1,2} \cup X_{1,3} \cup X_{2,3}]$ is $22$-regular.
\end{lemma}

\begin{proof}
 Let $v$ be a vertex of $X_{1,2}$. We count the number of $2$-paths from $v$ to $K_4$. Since it is adjacent to $2$ vertices of it, there must be $2\cdot 12 + 2\cdot 20$ such paths. Denote with $k$ the number of neighbors of $v$ in $X_2$. By the previous lemma, $v$ is adjacent to $2$ vertices in $X_0$, thus it is adjacent to $40-2-2-k$ vertices in $X_1$. Now we count 2-paths:
$$2\cdot 12 + 2\cdot 20=2k+1(40-2-2-k)+0\cdot 2+2\cdot 3\,.$$
 
Therefore, $v$ has $k=22$ neighbors in $X_2$,  and since it is not adjacent to $x_3$ it must have $20$ neighbors in $X_2 - X_{1,2} = N(x_3)$. This implies that $v$ has precisely 2 neighbors in $X_{1,2}$.
\end{proof}

In \cite{Makhnev} it was shown that if graph $X$ exists, then it does not have a 20-regular subgraph. Consider the induced subgraph on $X_2$ and remove the disjoint unions of cycles in 
$X_{1,2},X_{1,3},X_{2,3}$. We obtain a 20-regular subgraph and hence this configuration is impossible.

\subsection{Case $(1,34,54,2)$}

Let $X_0,X_1, X_2, X_3$ be the sets of vertices having $0,1,2$, and $3$ neighbors in $K_4$, respectively. In particular, let $x_0 \in X_0$ and $x_1\ne x_2 \in X_3$.

\begin{lemma} \label{lem:tab1}
Vertices $x_1$ and $x_2$ are not adjacent.
\end{lemma}

\begin{proof}
Suppose $x_1 \sim x_2$. There are up to isomorphism only two possible induced graphs on $K_4 \cup \{x_1,x_2\}$. Moreover, if we add the vertex $x_0$ we obtain 6 candidate graphs for an induced subgraph of $X$. None of them interlaces $X$, which was checked by a Sage program listed in Table \ref{table:sageprog}.
\end{proof}

\begin{lemma}\label{lem:b}
Vertex $x_0$ is adjacent to both vertices in $X_3$. Moreover, it has $2$ neighbors in $X_1$ and $36$ neighbors in $X_2$.
\end{lemma}

\begin{proof}
For the sake of contradiction, suppose $x_0$ is adjacent to $k \in \{0,1\}$ vertices of $X_3$. Let $t$ be the number of neighbors of $x_0$ in $X_1$. By double counting  $2$-paths from $x_0$ to $K_4$ we obtain: $$4\cdot 20 = 3k + t + 2(40-k-t)\,,$$ which gives that $k = t$. Without loss of generality suppose that $x_1$ is not adjacent to $x_0$. By counting the number of $2$-paths in a similar way we obtain that $x_1$ has $10$ neighbors in $X_2$. But by strong regularity, $x_0$ and $x_1$ must have $20$ common neighbors which is not possible since $x_0,x_1$ can share at most $k \leq 1$ common neighbors in $X_1$ and $10$ common neighbors in $X_2$. Hence $x_0$ is adjacent to both $x_1$ and $x_2$ and so $k = t = 2$ and the claim follows.
\end{proof}

In virtue of Lemma \ref{lem:b}, let $x_0',x_0''$ be the vertices in $X_1$ that are adjacent to $x_0$.

\begin{lemma}\label{lem:a}
Each vertex $x_1,x_2$ has $11$ neighbors in $X_2$.
\end{lemma}

\begin{proof}
The lemma follows by double counting 2-paths to $K_4$.
\end{proof}

\begin{lemma}\label{lem:d}
For $i=1,2$, the vertex $x_i$ is adjacent to at least one of the vertices in $\{x_0',x_0''\}$.
\end{lemma}

\begin{proof}
By Lemma \ref{lem:a}, the vertex $x_i$ has $11$ neighbors in $X_2$. Since $x_i$ is adjacent to $x_0$, by Lemma \ref{lem:b}, it has $12$ common neighbors with $x_0$. Thus it must be adjacent to at least one of $x_0', x_0''$.
\end{proof}

Let $X_2^{-0}$ be the set of vertices in $X_2$ that are not adjacent to $x_0$. Notice that by Lemma \ref{lem:b}, $|X_2^{-0}|=54-36=18$.

\begin{lemma}
At most one vertex from $X_2^{-0}$ is adjacent to $x_1$, and at most one is adjacent to $x_2$.
\end{lemma}

\begin{proof}
Vertex $x_1$ shares at most 2 common neighbors with $x_0$ in $X_1$ (possibly $x_0'$ or $x_0''$). Thus it must have at least 10 out of 11 neighbors (Lemma \ref{lem:a}) in $X_2$ adjacent to $x_0$. By symmetry, the same claim holds for $x_2$.
\end{proof}

\begin{lemma}
Each vertex in $X_2^{-0}$ that is not adjacent to any of the $\{x_1,x_2\}$,  has degree $t \leq 2$ in $X[X_2^{-0}]$, and it has precisely $t$ neighbors in $\{x_0',x_0''\}$. Vertices (at most two) in $X_2^{-0}$ that are adjacent to exactly one of the vertices in $\{x_1,x_2\}$ have degree $t-1$ in  $X_2^{-0}$ and $t\geq 1$ neighbors in $\{x_0',x_0''\}$. If there exists a vertex in $X_2^{-0}$ that is adjacent to both $x_1$ and $x_2$, then it is adjacent to both $x_0'$ and $x_0''$. In particular, such a vertex has degree 0 in $X[X_2^{-0}]$.
\end{lemma}

\begin{proof}
Notice that any $v\in X_2^{-0}$ must have 20 common neighbors with $x_0$. First, assume it is not adjacent to $x_1$ or $x_2$. Then their common neighbors can only be in $\{x_0',x_0''\}$, say $t$ of them, and in $X_2\backslash X_2^{-0}$. By double counting 2-paths from $v$ to $K_4$ we obtain that $v$ has 20 neighbors in $X_2$. Thus, precisely $t \leq 2$ of them must be in $X_2^{-0}$.

Second, assume that $v$ is adjacent to exactly one of the $x_1,x_2$. Then it has $20-1-t$ neighbors in $X_2\backslash X_2^{-0}$. On the other hand, by double counting, its degree in $X_2$ is $18$. Thus it's degree in $X_2^{-0}$ is $18-(19-t)=t-1$.

Finally, if $v$ is adjacent to $x_1$ and $x_2$, it has degree $16$ in $X_2$, thus all his neighbors have to be in $X_2\backslash X_2^{-0}$ and it also has to be adjacent to both vertices in $\{x_0',x_0''\}$. 
\end{proof}

\begin{lemma}
Each of the vertices $x_0',x_0''$ has $19-t$ neighbors in $X_2^{-0}$, where $t\in \{1,2\}$ is the number of its neighbors in $\{x_1,x_2\}$. Moreover, $x_0'$ and $x_0''$ are not adjacent.
\end{lemma}

\begin{proof}
By double counting $2$-paths from $x_0'$ to $K_4$ we have that $x_0$ has $31-2t$ neighbors in $X_2$.  Vertices $x_0$ and $x_0'$ have $12$ common neighbors. Let $s$ be equal to 1 if  $x_0'$ and $x_0''$ are adjacent and $0$ otherwise. Vertices $x_0$ and $x_0'$ must have $12-t-s$ common neighbors in $X_2$, thus $x_0'$ has $31-2t-(12-t-s)=19-t+s$ neighbors in $X_2^{-0}$. By symmetry the same holds for  $x_0''$ and since $|X_2^{-0}|=18$ it would imply that $x_0'$ and $x_0''$ have at least $16$ common neighbors in $X_2^{-0}$. Thus they are not adjacent and $s=0$. The lemma holds.
\end{proof}

The above lemmas give enough structure to be able to generate all graphs induced by $K_4\cup \{x_0,x_0',x_0'',x_1,x_2\}\cup X_{2}^{-0}$. 

\begin{prop} \label{prop:case134542}
There are $46$ graphs of the form  $K_4\cup \{x_0,x_0',x_0'',x_1,x_2\}\cup X_{2}^{-0}$ satisfying the structure described in this section.
\end{prop}

\subsection{Case $(2,31,57,1)$}

Let $X_0 = \{x_0,x_1\}, X_1, X_2, X_3 = \{x_3\}$ be the sets of vertices having $0,1,2$, and $3$ neighbors in $K_4$ respectively.

\begin{lemma} 
We have $x_0 \not \sim x_1$.
\end{lemma}

\begin{proof}
 If $x_0 \sim x_1$ then the number of $2$-paths from $x_0$ to $K_4$ is at most $3+2\cdot 38=79$. But since $x_0$ is not adjacent to any vertex of $K_4$, it should have precisely $4\cdot 20$ $2$-paths to it.
\end{proof}

\begin{lemma} 
$x_0 \sim x_3$ and $x_1 \sim x_3$. 
\end{lemma}

\begin{proof}
  Suppose $x_0$ is not adjacent to $x_3$ and let $N_0,N_1$, respectively, be the sets of neighbors of $x_0,x_1$ in $X_2$. By double counting 2-paths to $K_4$ from $x_0$ and $x_1$ we have $|N_0| = 40$ 
  while $|N_1|= 40-2t$ where $t \in \{0,1\}$ depending on whether $x_1$ is adjacent to $x_3$ or not. Since all the neighbors of $x_0$ are in $X_2$, we have $|N_0 \cap N_1| = 20$. But this implies $|N_0 \cup N_1| = 60-2t \geq 58$ which is not possible as $X_2$ has cardinality 57.
\end{proof}

\begin{lemma} 
Vertices $x_0$ and $x_1$ have precisely one neighbor in $X_1$. In particular, the two neighbors are distinct. Moreover, both $x_0$ and $x_1$ have 38 neighbors in $X_2$, 19 of them in common.
\end{lemma}

\begin{proof}
 Let $t$ be the number of neighbors of $x_0$ in $X_1$. By counting $2$-paths from $x_0$ to $K_4$ we have $$ 20 \cdot 4 = 3 + t + 2 \cdot (40-1-t)\,,$$ which gives that $t = 1$, and $x_0$ has 38 neighbors in $X_2$. Same holds for $x_1$. Let again $N_0,N_1$ be the sets of neighbors of $x_0,x_1$ in $X_2$. Then $|N_0 \cup N_1|\leq 57$, thus $|N_0 \cap N_1|\geq 19$. Since $x_0$ and $x_1$ have a common neighbor $x_3$, $|N_0 \cap N_1|= 19$ and $|N_0 \cup N_1| = 57$. In particular this implies that $x_0$ and $x_1$ cannot have a common neighbor in $X_1$.
\end{proof}

The last two lemmas now imply that $X_2$ can be partitioned into sets $X_2^0, X_2^{\{0,1\}}, X_2^{1}$ each of cardinality $19$ such that every vertex in $X_2^{i}$ is adjacent to $x_i$ and not adjacent to $x_{1-i}$ for $i = 0,1$ and every vertex in $X_2^{0,1}$ is adjacent to both $x_0$ and $x_1$.

Let us denote the neighbors of $x_0,x_1$ in $X_1$ by $x_0'$ and $x_1'$ respectively.

\begin{lemma}\label{lem:e}
If $x_3$ is adjacent to $x_1'$ then it has 1 neighbor in $X_2^0$, otherwise it has no neighbor in $X_2^0$.
\end{lemma}

\begin{proof}
Vertex $x_3$ has 12 neighbors in $X_2$. On the other hand, it must have 12 common neighbors with $x_1$. Notice that the common neighbors can only be in $X_2 \cup \{x_1'\}$. If $x_3$ is adjacent to $x_1'$, then it must have 11 neighbors in $X_2^{0,1} \cup X_2^1$, thus 1 in $X_2^0$. On the other hand, if $x_3$ is not adjacent to $x_1'$, then it must have all 12 neighbors in $X_2^{0,1} \cup X_2^1$ and no neighbor in $X_2^0$.
\end{proof}

\begin{lemma}
Each vertex in $X[X_2^0]$ has maximal degree 2. If $v \in X_2^0$ has degree 2, then it is not adjacent to $x_3$, but it is adjacent with $x_1'$. If it has degree 1, then it is adjacent with ether both $x_3$ and $x_1'$ or none of them. Finally, if $v$ has degree 0, then it is adjacent to $x_3$ and not adjacent with $x_1$.
\end{lemma}

\begin{proof}
Pick a vertex $v\in X_2^0$ and let $s\in \{0,1\}$ indicate whether $v$ is adjacent with $x_1'$. Similarly, let $t\in \{0,1\}$ indicate if $v$ it is adjacent with $x_3$. Let $r$ be the number of neighbors of $v$ in $X_2$. We count the number of 2-paths from $v$ to $K_4$:
$$2\cdot 12 +2 \cdot 20 = 2\cdot 3 + 3t + 2r + (40-2-t-r-1)\,. $$
Thus $r=21-2t$. Vertex $v$ and $x_1$ must have $20$ common neighbors. This implies that the number of neighbors of $v$ in $X_2^{0,1} \cup X_2^1$ is $20-t-s$. We have that $v$ has $(21-2t)-(20-t-s)=1-t+s$ neighbors in $X_2^0$ which proves our lemma.
\end{proof}

By generating all possible graphs of the form  $X_2^0\cup \{x_0,x_1,x_1',x_3\} \cup K_4$ we infer

\begin{prop} \label{prop:case231571}
None of the possible graphs on $X_2^0\cup \{x_0,x_1,x_1',x_3\} \cup K_4$ interlaces $X$.
\end{prop}

\subsection{Case $(0,37,51,3)$}

Let again $X_1,X_2$ and $X_3 = \{x_0,x_1,x_2\}$ be the respective subsets of vertices of $X$ with 1,2,3 neighbors in $K_4$. There are three non-isomorphic ways to introduce $3$ vertices to $K_4$ by joining each vertex to 3 vertices of $K_4$. Each such graph $G_1,G_2,G_3$ can be uniquely described by a tuple $\vec{n} = (n_1,n_2,n_3,n_4)$ counting the number of edges from the $i'$th vertex of $K_4$ to $X_3$. By relabeling the vertices of $K_4$ if needed we obtain three tuples $(0,3,3,3), (1,3,3,2)$ and $(2,2,2,3)$ which we will cover as subcases. In what follows we first establish certain structural claims about the configuration $(0,37,51,3)$.

\begin{lemma}  \label{lem:tab4}
The vertices of $X_3$ form an independent set.
\end{lemma}

\begin{proof} 
No matter how we introduce edges among the vertices of $X_3$ in the graph $K_4 \cup X_3$ we do not obtain a graph interlacing $X$. 
\end{proof}

\begin{lemma}\label{lem:aa} 
Every vertex $x\in \{x_0,x_1,x_2\}$ has 27 neighbors in $X_1$ and 10 neighbors in $X_2$. 
\end{lemma}

\begin{proof} Let $x$ have $k$ neighbors in $X_1$ and $l=40-3-k$ neighbors in $X_2$. We count the number of paths of length 2 from $x$ to $K_4$. We have $3\cdot 12+20=3\cdot 3+k+2 \cdot (40-k-3)$ and thus $k=27$ and $l=10$.
\end{proof}

Let $X_2^0,X_2^1, X_2^3$ be the neighbors in $X_2$ of $x_0,x_1,x_2$ respectively. We have proved that  $|X_2^0| = |X_2^1| = |X_2^2| = 10$. Notice that the sets $X_2^0,X_2^1,X_2^2$ need not be disjoint.

\begin{lemma}\label{lem:ab} 
It holds $ |X_2^0 \cap X_2^1 |, | X_2^0 \cap X_2^2|,| X_2^1 \cap X_2^2|  \in \{0,1\}$.
\end{lemma}

\begin{proof}
Vertices $x_0$ and $x_1$ have 20 common neighbors, at least 2 of them are on $K_4$. Each of $x_0,x_1$ has $27$ neighbors in $X_1$, where $|X_1|=37$. Thus they must have at least 17 common neighbors in $X_1$. The latter implies that they have at most one common neighbor in $X_2$. By symmetry the claim now holds for $| X_2^0 \cap X_2^2|$ and $| X_2^1 \cap X_2^2|$.
\end{proof}

\begin{lemma}\label{lem:notiang} 
For $i \in \{1,2\}$ the graph induced by $X_2^0 \setminus X_2^i$ is triangle-free. 
\end{lemma}

\begin{proof}
Assume that there exists a triangle in $X_2^0 \setminus X_2^i$. Together with $x_0$ it forms a $K_4$. Assume this $K_4$ does not extend to a $K_5$. Vertex $x_i$ is not adjacent to any of the vertices of this $K_4$. Thus we have a case of $K_4$ with some vertices that are not adjacent to it. We have already shown that this is not possible. On the other hand, if $K_4$ extends to a $K_5$, we have a $K_5$ and a vertex that is adjacent to at most one vertex on it. This is impossible, since if $X$ has a $5$-clique $K_5$ then using the formula of \cite{bono} for $K_5$, as we did for $K_4$, gives us that every vertex in $V(X) \setminus K_5$ has precisely two neighbors in $K_5$.
\end{proof}

Let $X_1^{-0}$ denote the subgraph of $X_1$ induced on all the vertices not adjacent to $x_0$ and let $X_1^{i}$ be the set of vertices in $X_1$ adjacent to $x_i$ for $i\in \{0,1,2\}$.

\begin{lemma}\label{lem:ac} 
The graph $X_1^{-0}$ has 10 vertices. At most one of the vertices in $X_1^{-0}$ is not adjacent to $x_1$ and at most one is not adjacent to $x_2$. Moreover, the graph $X_1^{-0}\setminus X_1^i$, for $i\in \{1,2\}$, has no triangles. Each vertex $v\in X_1^{-0}$ has degree $k$ in $X_1^{-0}$ at most 2 and is adjacent to precisely $11-t-m+k$ vertices in $X_2^0$, where $t\in \{0,1,2\}$ is the number of vertices adjacent to $v$ in $\{x_1,x_2\}$ and $m\in \{0,1\}$ the number of common neighbors of $v$ and $x_0$ on $K_4$.  
\end{lemma}

\begin{proof}
By Lemma \ref{lem:aa}, $|X_1^0|=27$, thus $|X_1^{-0}|=10$. For $i\in \{1,2\}$, $x_i$ and $x_0$ share at least 2 neighbors on $K_4$. Since they are non-adjacent, they share 20 neighbors, thus at most 18 in $X_1$. This implies $|X_1^0\cap X_1^i|\leq 18$, thus  $|X_1^0\cup X_1^i| \geq 27+27-18=36$. Since $|X_1|=37$ we see that there exist at most one vertex in $X_1$ that is not adjacent to $x_0$ and not to $x_i$.

If there is a triangle in the graph  $X_1^{-0}\setminus X_1^i$, this triangle forms a $K_4$ with $x_0$, while $x_i$ is not adjacent to any of its vertices. If this $4$-clique is not a part of a $5$-clique the assertion follows since this case has already been dealt with (a $K_4$ with some vertices not adjacent to it). On the other hand, if this $K_4$ is a part of a $K_5$, we have an induced subgraph of $K_5$ together with a vertex that is adjacent to one or none of the vertices on $K_5$. But every vertex in $V(X) \setminus V(K_5)$ has two neighbors in $V(K_5)$ as in Lemma \ref{lem:notiang}. A contradiction. Hence $X_1^{0,2}$ is indeed triangle-free. 

Let now $v\in X_1^{-0}$ be as in the lemma. Denote with $j$ the number of its neighbors in $X_1$. By counting 2-paths to $K_4$ we get
$$12+3\cdot 20= 3+3t+j+2(40-1-t-j),$$
hence $j=9+t$. Denote with $l$ the number of neighbors of $v$ in $X_2^0$. Vertex $v$ and $x_0$ have 20 common neighbors, thus
$$20=(j-k)+l+m=(9+t-k)+l+m,$$
from which we get $11-t-m+k=l\leq 10$. This implies also that $k\leq 2$.
\end{proof}

\begin{lemma}\label{lem:notiang2} 
Let $v$ be a vertex in $X_2^0$ and $k$ the number of its neighbors in $X_2^0$, $t\in\{0,1,2\}$ the number of its neighbors in $\{x_1,x_2\}$, and $m\in\{1,2\}$ the number of common neighbors of $v$ and $x_0$ on $K_4$. Then:
$$k+m+t\leq 3.$$
In particular $k\leq 2$.
\end{lemma}

\begin{proof}
Let $v\in X_2^0$. Denote  with $j$ the number of neighbors of $v$ in $X_1$. By counting 2-paths from $v$ to $K_4$ we get:
$$2\cdot 12+2\cdot 20 = 2\cdot 3+1\cdot 3+ 3t +j +2(40-2-1-t-j),$$
thus $j=19+t$. Let now $l\leq 10$ be the number of neighbors of $v$ in $X_1^{-0}$. Vertices $v$ and $x_0$ have 12 common neighbors:
$$12=k+m+(j-l)=k+m+(19+t-l),$$
thus $7+k+m+t=l\leq 10$ and $k+m+t\leq 3$. Since $m\in\{1,2\}$, $k\leq 2$.
\end{proof}

By generating all graphs induced on $K_4\cup \{x_0,x_1,x_2\} \cup X_2^0 \cup X_1^{-0}$ we infer.

\begin{prop} \label{prop:case037513}
There are $157$ graphs of the form $K_4\cup \{x_0,x_1,x_2\} \cup X_2^0 \cup X_1^{-0}$ that interlace $X$.
\end{prop}

The case analysis carried out in this section resulted in $203$ graphs, $190$ of them having a star complement as an induced subgraph. The remaining $13$ had a subgraph of order $19$ with no $2$ as eigenvalue. Such subgraphs were extended in all possible ways  and a final list of $3225$ star complements was obtained. By computing the respective comparability graphs and their clique numbers we have determined that none of them has clique number larger than $74$. Hence we deduce Proposition \ref{prop:clnum}.

\begin{table}[h]
    \begin{tabular}{| l | l | l |} 
    \hline
    Claim & Program & Output \\ \hline
    Lemma \ref{lem:tab1} &  K4/134542/Claim1.sage & \\
    Proposition \ref{prop:case134542} & K4/134542/Case1.sage & case13452.g6 \\
    Proposition \ref{prop:case231571} & K4/231571/Case2.sage & \\
    Lemma \ref{lem:tab4} & K4/037513/Claim2.sage & \\
    Proposition \ref{prop:case037513} & K4/037513/generateFinal.sage & case037513.g6 \\ 
    Lemma \ref{lem:tab14} & K5/extendTriangle.sage & triangles.g6 \\
    \hline
    \end{tabular}
    \caption{Sage programs}\label{table:sageprog}
\end{table}

\section{Main result}
Let $K_5$ be a $5$-clique of $X$ with vertex set $\{k_1,\ldots,k_5\}$. As already mentioned in Lemma \ref{lem:notiang}, results from \cite{bono} imply that every vertex of $X$ that is not in $K_5$ has precisely two neighbors in $K_5$. For $1\leq i<j\leq 5$, let $X_{i,j}$ be vertices in $V(X)\backslash V(K_5)$ that are adjacent to $k_i$ and $k_j$. Since $k_i$ and $k_j$ are adjacent, they must have 12 common neighbors, 3 of them already on $K_5$. Hence $V(G) \setminus V(K_5)$ is partitioned into 10 sets of $9$ vertices, namely $X_{0,1},X_{0,2},\ldots, X_{4,5}$. In what follows we establish structural results about these partitions.

\begin{lemma}
For any $1 \leq i < j \leq 5$ the graph $X_{i,j}$ is either $\overline{K_9}$ or $K_3 \cup \overline{K_6}$ or $K_3 \cup K_3 \cup \overline{K_3}$ or $K_3 \cup K_3 \cup K_3 $.
\end{lemma}

\begin{proof}
Assume there exists an edge $e = \{x,y\}$ in the graph $X_{i,j}$. Then the vertices $\{x,y,k_i,k_j\}$ induce a 4-clique. By the result of the previous section, every 4-clique is contained in a  5-clique. Clearly, the additional vertex must be in $X_{i,j}$. Hence we have proved that every edge $e$ in $X_{i,j}$ is contained in a triangle in $X_{i,j}$. Let $T$ be a triangle in $X_{i,j}$ and $v\in X_{i,j}$ a vertex not on $T$. Since $T \cup \{k_i,k_j\}$ induces a 5-clique, every vertex not on this 5-clique is adjacent to exactly 2 vertices on this clique. Since $v$ is adjacent to $k_i$ and $k_j$, it is not adjacent to $T$ and the lemma follows.
\end{proof}

As it turns out every pair of triangles in distinct partitions $X_{i,j}, X_{k,l}$ induce quite a regular structure.

\begin{lemma}\label{lem:tri1}
Let $1 \leq i  < j \leq 5$, $1 \leq k < l \leq 5$ and let $T, T'$ be two triangles of $X_{i,j}$ and $X_{k,l}$, respectively. Let $c = |\{i,j,k,l\}|$. If $c = 3$, then the edges from $T$ to $T'$ form a perfect matching. If $c = 4$, they form a complement of a perfect matching.
\end{lemma}

\begin{proof}\label{lem:tri2}
First assume $c=3$. Since $T \cup \{k_i,k_j\}$ forms a 5-clique, every vertex of $T'$ is adjacent to exactly 2 vertices in this 5-clique. Since $c=3$, it must be adjacent to exactly one vertex in $T$. Similarly, every vertex of $T$ must be adjacent to exactly one vertex in $T'$. Thus, the edges from $T$ to $T'$ form a perfect matching. The case when $c=4$ is similar.
\end{proof}

Our next lemma shows that not all partitions $X_{i,j}$ contain a triangle. In fact at most 7 do.

\begin{lemma} \label{lem:tab14}
There are at least three distinct pairs $\{i,j\},\{k,l\},\{m,n\}$ such that $X_{i,j},X_{k,l}$ and $X_{m,n}$ are independent sets of $X$.
\end{lemma}

\begin{proof}
Using Lemma \ref{lem:tri1} and Lemma \ref{lem:tri2} we wrote a Sage program generating all possible graphs on $\{k_1,\ldots,k_5\}$ and $8$ triangles, each contained in a different set $X_{i,j}$, for $1\leq i<j \leq 5$. There are $2$ non isomorphic ways to chose $8$ sets for triangles (among $10$ sets). All the obtained graph in one configuration do not interlace $X$, while in the other configuration only one graph was found giving rise to $209$ comparability graphs. None of them has a clique of order $75$. Thus the lemma follows.
\end{proof}

\begin{theorem}\label{thm:main}
The graph $X$ does not exist.
\end{theorem}

\begin{proof}
With the help of McKay's program {\em genbg} we generated all graphs of the form $X_{1,2} \cup X_{2,3} \cup K_5$, assuming $X_{1,2}$ and $X_{2,3}$ are independent sets. By constructing star complements from the obtained graphs we obtained a list of $3998479$ graphs. By checking the respective comparability graphs it turns out that none of the them has a $75$ clique, hence proving our claim. 
\end{proof}

\section{Final remarks}

The presented arguments for the non-existence of a $(95,40,12,20)$ are almost the same as the one used in showing the non-existence of a $(75,32,10,16)$ SRG. What is interesting is that in the case of a $(95,40,12,20)$ SRG the computational aspect is significantly smaller. Indeed it takes only about 2 weeks of CPU time on a standard desktop machine to run the computational parts of the presented result. The two main reasons for this is that the target star complements of $(95,40,12,20)$ have small order (relative to the order of $X$). Finally, when computing the respective clique numbers it is much easier to prove that the comparability graphs do not have a clique of order $75$ as opposed to $56$ in the case of a $(75,32,10,16)$ SRG.


\bibliographystyle{plain}
\bibliography{biblio}

\begin{thebibliography}{10}

\bibitem{GitHub}
J.~Azarija.
\newblock {\em Jazarija github page}.
\newblock {\tt https://github.com/jazarija/}.

\bibitem{self-cite}
J.~{Azarija} and T.~{Marc}.
\newblock {There is no $(75,32,10,16)$ strongly regular graph}.
\newblock {\em ArXiv e-prints}, September 2015.

\bibitem{Behbahani}
M.~Behbahani and C.~Lam.
\newblock Strongly regular graphs with non-trivial automorphisms.
\newblock {\em Discrete Math.}, 311:132--144, 2011.

\bibitem{bono}
A.V. Bondarenko, A.~Prymak, and D.~Radchenko.
\newblock Non-existence of $(76, 30, 8, 14)$ strongly regular graph and some
  structural tools.
\newblock {\em arXiv preprint arXiv:1410.6748}, 2014.

\bibitem{Brower2}
A.~E. Brouwer and J.~H. van Lint.
\newblock Strongly regular graphs and partial geometries.
\newblock {\em Enumeration and design (Waterloo, Ont., 1982)}, pages 85--122,
  1984.

\bibitem{brouwer3}
Andries~E. Brouwer and Willem~H. Haemers.
\newblock {\em Spectra of graphs}.
\newblock Universitext. Springer, New York, 2012.

\bibitem{Brower}
A.~E. Brower.
\newblock {\em Parameters of {S}trongly {R}egular {G}raphs}.
\newblock {\tt http://www.win.tue.nl/~aeb/graphs/srg/srgtab.html}.

\bibitem{cohen-2016}
N.~{Cohen} and D.~V. {Pasechnik}.
\newblock {Implementing Brouwer's database of strongly regular graphs}.
\newblock {\em ArXiv e-prints}, January 2016.

\bibitem{Star}
D.~M. Cvetkovi{\'c}, P.~Rowlinson, and S.~Simic.
\newblock {\em Eigenspaces of Graphs}.
\newblock Number~66. Cambridge University Press, 1997.

\bibitem{degr}
J.~Degraer.
\newblock {\em Isomorph-free exhaustive generation algorithms for association
  schemes}.
\newblock PhD thesis, Ghent University, 2007.

\bibitem{godsil2001algebraic}
C.~D. Godsil and G.~Royle.
\newblock {\em Algebraic {G}raph {T}heory}, volume 207.
\newblock Springer New York, 2001.

\bibitem{haemers}
W.~H. Haemers.
\newblock There exists no {$(76,21,2,7)$} strongly regular graph.
\newblock In {\em Finite geometry and combinatorics ({D}einze, 1992)}, volume
  191 of {\em London Math. Soc. Lecture Note Ser.}, pages 175--176. Cambridge
  Univ. Press, Cambridge, 1993.

\bibitem{Makhnev}
A.~A. Makhnev.
\newblock On strongly regular graphs with parameters $(75,32,10,16)$ and
  $(95,40,12,20)$.
\newblock {\em Fundam. Prikl. Mat.}, 6:179--193, 2000.

\bibitem{Milosevic}
M.~Milo{\v{s}}evi{\'c}.
\newblock An example of using star complements in classifying strongly regular
  graphs.
\newblock {\em Filomat}, 22:53--57, 2008.

\end{thebibliography}

\end{document}